\documentclass[reqno,12pt,letterpaper]{amsart}
\usepackage{amsmath,amssymb,amsthm,graphicx,mathrsfs,url}
\usepackage[usenames,dvipsnames]{color}
\usepackage[colorlinks=true,linkcolor=Red,citecolor=Green]{hyperref}
\usepackage{amsxtra,esint}
\usepackage{mathtools}
\usepackage{MnSymbol}
\mathtoolsset{showonlyrefs,showmanualtags}

\usepackage{float}
\usepackage[font=small,labelfont=bf]{caption}
\usepackage{bm}

\usepackage{imakeidx}
\makeindex

\usepackage[square,sort,comma,numbers]{natbib}

\def\?[#1]{\textbf{[#1]}\marginpar{\Large{\textbf{??}}}}

\def\smallsection#1{\smallskip\noindent\textbf{#1}.}
\let\epsilon=\varepsilon 

\newtheorem*{theorem*}{Theorem}
\newtheorem{theorem}{Theorem}
\newtheorem{prop}{Proposition}[section]
\newtheorem{defi}[prop]{Definition}

\newtheorem{lemma}[prop]{Lemma}

\newtheorem{rem}{Remark}[section]

\newtheorem{claim}{Claim}[section]
\numberwithin{equation}{section}

\DeclareMathOperator{\Spec}{Spec}

\let\Re=\Real

\DeclareMathOperator{\supp}{supp}
\DeclareMathOperator{\vol}{vol}

\newcommand*{\dd}{\mathop{}\!\mathrm{d}}

\DeclareMathOperator{\dist}{dist}

\DeclareMathOperator{\Tr}{Tr}
\newcommand{\ip}[2]{\left \langle#1,#2 \right \rangle}
\newcommand{\mat}[1]{\begin{pmatrix} #1 \end{pmatrix}}

\newcommand{\set}[1]{ \left \{ #1 \right \}}

\newcommand{\N}{\mathbb{N}}
\newcommand{\p}{\partial}
\newcommand{\dbar}{\overline{\p}}
\newcommand{\Z}{\mathbb{Z}}
\newcommand{\T}{\mathbb{T}}
\newcommand{\E}{\mathbb{E}}
\newcommand{\e}{\varepsilon}

\newcommand{\R}{\mathbb{R}}

\newcommand{\norm}[1]{ \left \| #1 \right \| }

\newcommand{\C}{\mathbb{C}}
\renewcommand{\phi}{\varphi}
\renewcommand{\P}{\mathbb{P}}

\renewcommand{\Re}[1]{{\rm{Re}} \left ( #1\right ) }

\usepackage{graphicx}

\title{A Probabilistic Weyl-Law for Perturbed Berezin-Toeplitz Operators}
\author{Izak Oltman} 
\address[Izak Oltman]{Department of Mathematics, University of California, Berkeley, CA 94720}
\email{izak@math.berkeley.edu}
\date{\today}

\begin{document}

\begin{abstract}
This paper proves a probabilistic Weyl-law for the spectrum of randomly perturbed Berezin-Toeplitz operators, generalizing a result proven by Martin Vogel in \citep{Vogel}. This is done following the strategy of \citep{Vogel} using the exotic symbol calculus developed by the author in \citep{Izak}.
\end{abstract}

\maketitle

\section{Introduction}
This paper generalizes a result of Martin Vogel in \citep{Vogel} which proves a probabilistic Weyl-law for quantizations of functions on tori. Here we do the same, but with the tori replaced by arbitrary K\"ahler manifolds equipped with positive line bundles. 

In \citep{Vogel}, Vogel considers Toeplitz quantizations of smooth functions on a real $2d$-dimensional torus, which associates every smooth function $f$ on the torus to a family of $N^d \times N^d$ matrices, $f_N$, for all $N \in \N$ (here $N^{-1}$ is the semi-classical parameter). A recent physical motivation for such constructions is written by Deleporte in \citep[Section 1]{Deleporte}. Next, a random matrix with sufficiently small norm is added to $f_N$, and the spectrum is shown to obey an almost-sure Weyl-law as $N$ goes to infinity. This was conjectured by Christiansen and Zworski in \citep{Christiansen} and is a major extension of their work. 

This result is most striking when the unperturbed matrix is non-self-adjoint. For example, if $f (x) = \cos (2\pi x) + i \cos ( 2\pi\xi)$, then the quantization is
\begin{align}f_N = 
\begin{pmatrix}
\cos (2\pi /N) & i/2 & 0 & 0 & \cdots & i/2 \\ 
i/2 & \cos(4\pi / N) & i/2 & 0 & \cdots & 0 \\ 
0 & i/2 & \cos(6\pi /N) & i/2 & \ddots & 0 \\ 
\vdots & \ddots & \ddots & \ddots & \ddots & \vdots \\ 
0 & \cdots & 0 & i/2 & \cos(2(N-1)\pi/N) & i/2 \\ 
i/2 & 0 & \cdots & 0 & i/2 & \cos(2\pi)
\end{pmatrix},
\end{align}
which numerically has spectrum contained on two crossing lines in the complex plane. This operator is aptly named the Scottish flag operator and is further described by Embree and Trefethen in \citep{Trefethen}. Interestingly, (as far as we are aware) it is unknown analytically where the spectrum of $f_N$ lives. However, if randomly perturbed, the spectrum spreads out with density given by the push-forward of the Lebesgue measure on the torus by $f$. Figure \ref{fig.1} plots the spectrum of $f_N$ with no perturbation, and with a small perturbation.
\begin{figure}[h!]
\includegraphics[width = \textwidth]{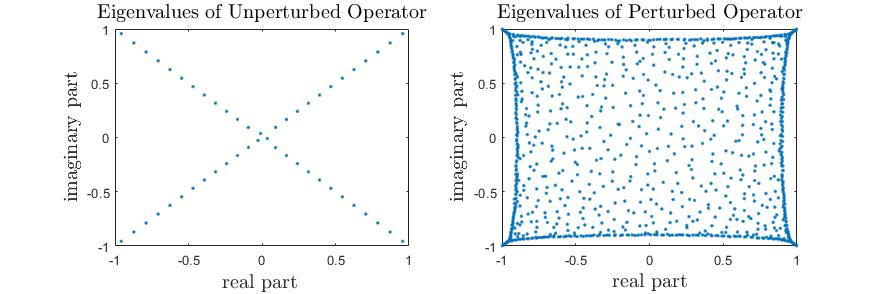}
\caption{Left: Eigenvalues of the Scottish flag operator with $N = 50$. Right: Eigenvalues of the Scottish flag operator with a small random perturbation with $N = 1000$. \label{fig.1}}
\end{figure}

The spectral properties of randomly perturbed non-self-adjoint operators was pioneered by Hager in \citep{Hager}, in which the operator $hD_x + g(x) : L^2 (S^1 ) \to L^2 (S^1)$ was studied. This result, and numerous subsequent results are discussed by Sj\"ostrand in \citep{Sjo_book}. There are related results describing spectral properties of randomly perturbed Toeplitz matrices, which can be defined as quantizations of symbols on $\T^2$ with symbol independent of $x$. See Davies and Hager \cite{davies2009}, Guionnet, Wood and Zeitouni \cite{wood2014}, Sj\"ostrand and Vogel \cite{sjostrand2021} \cite{sjostrand20212}, and references given there.

This paper is the natural generalization of Vogel's result in \citep{Vogel}. Here we prove a similar result for quantizations of functions on K\"{a}hler manifolds (with sufficient structure, as discussed in Section \ref{section:setup}). These quantizations, called Berezin-Toeplitz operators (or just Toeplitz operators) were first described by Berezin in \citep{Berezin} as a particular type of quantization of symplectic manifolds. Following \citep{Berezin}, for every smooth function $f$ on a quantizable K\"ahler manifold $X$, we get a family of finite rank operators, $T_N f$, indexed by $N \in \N$ (see \citep{Rouby} for a connection between these quantizations, and quantizations on the torus) which have physical interpretations. Deleporte in \citep[Appendix A]{Deleporte} relates this quantization to spin systems in the large spin limit, and Douglas and Klevtsov in \citep{Douglas} use path integrals for particles in a magnetic field to derive the Bergman kernel (a key ingredient in constructing $T_N f$).

Next, if we add a small Gaussian-type random perturbation $\mathcal{G}_\omega$ to these operators (see Definition \ref{def:random}), the empirical measures weakly converge almost surely (see Theorem \ref{thm.main} in Section \ref{section:setup} for a precise statement). Theorem \ref{theorem:general perturbations} states a result about more general random perturbations $\mathcal{W}_\omega$ (see Definition \ref{def:random}) but with a more restrictive coupling constant. A consequence of Theorem \ref{theorem:general perturbations} is the following probabilistic Weyl-law.

\begin{theorem}[\textbf{A Probabilistic Weyl-law}]\label{thm:weyl}
Given a quantizable K\"ahler manifold $X$, $f\in C^\infty(X;\C)$ such that there exists $\kappa \in (0,1]$ so that
\begin{align}
\mu_d (\set{x \in X : |f(x) - z|^2 \le t }) = \mathcal{O}(t^\kappa)
\end{align}
as $t\to 0$ uniformly for $z \in \C$ (where $\mu_d$ is the Liouville volume form on $X$), $\mathcal{ W}_\omega$ a random matrix (see Definition \ref{def:random}), and $\Lambda \subset \C$. Then almost surely
\begin{align}
\left ( \frac{2\pi}{N} \right ) ^d \# \set{\Spec (T_N f + N^{-d} \mathcal{W}_\omega ) \cap \Lambda} \xrightarrow{N\to \infty} \mu_d (x \in X : f(x) \in \Lambda).
\end{align}
\end{theorem}

Finer results are expected for describing the spectrum of randomly perturbed Toeplitz operators. In \citep{Vogel}, precise statements about the number of eigenvalues are obtained using counting functions of holomorphic functions. Here we only show weak convergence of the empirical measures, but achieve this in a relatively simple way using logarithmic potentials as presented in \citep{Sjo-vogel}.

Here we present numerical examples to motivate the main result of this paper. Consider the K\"{a}hler manifold $\C \P^1$ (complex protective space of dimension $1$) which can be identified with the real $2$-sphere with coordinates $(x_1,x_2,x_3)$. In Figure \ref{fig.2}, we compute the spectrum of the quantization of the function $f = x_1 + 2x_2^2+ i x_2$. Before perturbation, the spectrum lies on several lines in the complex plane, somewhat analogous to the Scottish flag operator. However, as a perturbation is added, the spectrum fills in. This paper describes the structure of the spectrum of this perturbed operator in the semiclassical limit, as $N \to \infty$.

\begin{figure}[h!]
\centering
\includegraphics[width = \textwidth]{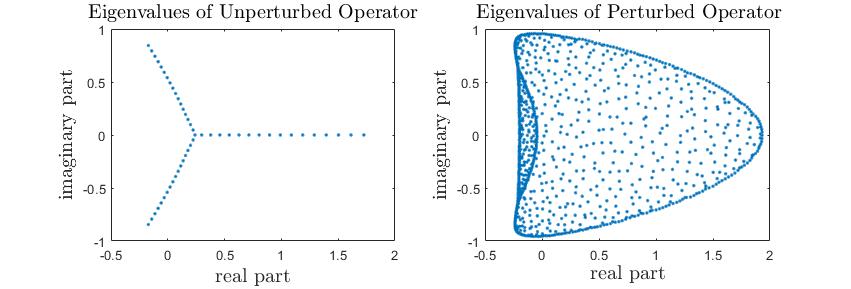}
\caption{Left: Eigenvalues of the Toeplitz operator on $\C \P^1$ identified with the real $2-$sphere with symbol $x_1 + 2 x_1^2 + ix_2$ and $N = 50$. Right: Eigenvalues of the same operator but with a small random perturbation and $N = 1000$.\label{fig.2}}
\end{figure}

Numerical verification of this paper's result can be seen if $f = i x_1 +x_2$ (still on $\C \P^1$). Figure \ref{fig.3} computes the spectrum of $T_Nf $ with a random perturbation added, and plots the number of eigenvalues in circles of increasing radii versus the predicted number of such eigenvalues by Theorem \ref{thm:weyl}. More animations can be found on my \href{https://math.berkeley.edu/~izak/research/toeplitz/movies.html}{website}\footnote{ \url{https://math.berkeley.edu/~izak/research/toeplitz/movies.html}}.

\begin{figure}[h!]
\includegraphics[width = \textwidth]{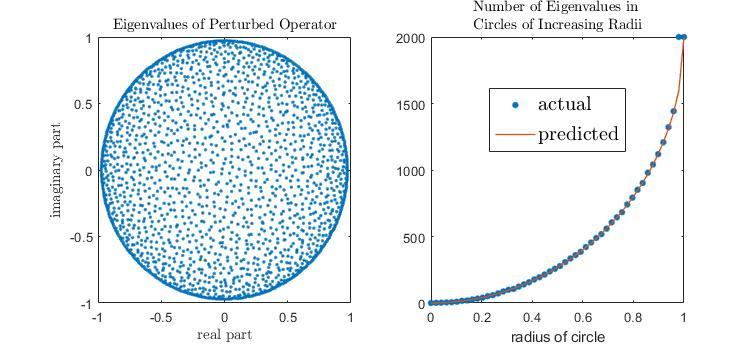}
\caption{Left: Eigenvalues of the randomly perturbed Toeplitz operator on $\C \P^1$ identified with the real $2-$sphere with symbol $ix_1 + x_2$ an $N = 2000$. Right: The number of eigenvalues within circles in the complex plane centered at zero with radii ranging from $0$ to $1$, plotted against the predicted distribution of eigenvalues from Theorem \ref{thm:weyl}.\label{fig.3}}
\end{figure}

\smallsection{Outline of Paper} Section \ref{section:setup} reviews background material and states the main result of this paper (Theorem \ref{thm.main}). In Section \ref{results}, a series of preliminary results about Toeplitz operators are presented. Section \ref{log potential} reviews logarithmic potentials and reduces Theorem \ref{thm.main} to proving a probabilistic bound involving logarithmic derivatives of Toeplitz operators. Section \ref{grushin approach} sets up a Grushin problem to further reduce the problem to prove probabilistic bounds on spectral properties of self-adjoint operators. Section \ref{section b1} proves a deterministic bound involving the logarithmic derivative of Toeplitz operators. The technique involves scaling the symbol by a power of $N$, and therefore relies on the exotic calculus presented in Section \ref{results}. Finally, Section \ref{section:last} chooses constants to establish the required probabilistic bound for the almost sure convergence in Theorem \ref{thm.main}. In Section \ref{section:general random perturbations}, we describe how to extend this result to the more general random perturbations as stated in Theorem \ref{theorem:general perturbations}.

\smallsection{Notation} We will use the following notation in this paper for functions $f$ and $g$ depending on $N$. We write $f = \mathcal{O}(g)$ if there exists $C> 0$ independent of $N$ such that $|f| \le C g$. We write $f = \mathcal{O}(N^{-\infty})$ if for every $M \in \N$, $f = \mathcal{O}(N^{-M})$. Any subscript in the big-O will denote dependence of $C$ of what is in the subscript. We will write $f \lesssim g$ if there exists a $C>0$ independent of $N$ such that $f \le C g$. We write $f\ll g$ to mean that $C f \le g$ for some sufficiently large $C> 0 $ independent of $N$. For a $u,v,w$ elements of a Hilbert space, denote $u\otimes v$ the map that sends $w$ to $u \ip{w}{v}$.
 
\section{Main Result}\label{section:setup}

Let $(X,\sigma)$ be a compact, connected, $d-$dimensional K\"{a}hler manifold with a holomorphic line bundle $L$ with positively curved Hermitian metric locally given by $h = e^{-\phi}$. That is over each fiber $x\in X$, $\norm{v}_{h} := e^{-\phi(x)} |v|$. Given this, the globally defined symplectic form, $\sigma$, is related to the Hermitian metric by $i\p \dbar \phi = \sigma$. Fixing local trivializations, $\phi$ can be described as a strictly plurisubharmonic smooth real-valued function (called the K\"{a}hler potential). This is further outlined by Le Floch in \citep{Floch}.

Let $L^N$ be the $N$th tensor power of $L$, which has Hermitian metric $h_N := e^{-N\phi}$. Let $\mu_d = \sigma ^{\land d} / d!$ be the Liouville volume form on $X$. This provides an $L^2$ structure on sections of $L^N$. Indeed, if $u$ and $v$ are smooth sections on $L^N$, then define
\begin{align}
\ip{u}{v} _{L^N} := \int_X h_N (u,v) \dd\mu_d.
\end{align}
Define $L^2 (X,L^N)$ to be the space of smooth sections of $L^N$ with finite $L^2$ norm. In this $L^2$ space, let $H^0(X,L^N)$ be the space of holomorphic sections. 
\begin{prop} \label{prop:dimension}
The dimension of $H^ 0(X,L^N)$ is finite, and is asymptotically
\begin{align}
\left (\frac{N}{ 2\pi} \right )^d \vol(X) + \mathcal{O} (N^{d-1}) .
\end{align}
\end{prop}
\begin{proof}
See \citep[Corollary 2]{Charles}.
\end{proof}
For the remainder of this paper, denote $\dim (H^0 ,(X,L^N))$ by $\mathcal{ N} = \mathcal{N}(N)$. The orthogonal projection from $L^2(X,L^N)$ to $H^0(X,L^N)$ is called the Bergman projector and is denoted by $\Pi_N$. Finally, given $f\in C^\infty (X;\C)$, the Toeplitz operators associated to $f$, written $T_{N}f$, are defined for each $N \in \N$ as $T_{N}f (u) = \Pi_N(fu)$, where $u\in H^0 (X,L^N)$. In this way, $T_{N}f$ are finite rank operators mapping $H^0 (X,L^N)$ to itself. For the remainder of this paper, we will fix a basis for $H^0 (X,L^N)$ so that $T_Nf$ (and similar operators) can be considered as matrices.

The class of functions to quantize will often depend on $N$. To define this symbol class requires local control of functions. Fix a finite atlas of neighborhoods $(U_i , \zeta_i ) _{i \in \mathcal{I}}$ for the K\"{a}hler manifold $X $.
\begin{defi}[$\bm{S(1)}$]
$S(1)$ is the set of all smooth functions $f$ on $X$ taking complex values which can be written asymptotically $f \sim \sum N^{-j} f_j$, where $f_j \in C^\infty (X;\C)$ do not depend on $N$. This tilde means that for all $\alpha \in \N$
\begin{align}
\p^\alpha_x \left (f \circ \zeta_i(x) - \sum _{j=0}^{M} N^{-j} f_j \circ \zeta _i (x) \right ) = \mathcal{O} _\alpha (N^{-j-1}) 
\end{align}
for all $i\in \mathcal{I}$, and all $\alpha\in \N^d$. By Borel's theorem, given any $f_j \in S(1)$ not depending on $N$, there exists $f \in S(1)$ such that $f \sim \sum N^{-j} f_j$.
\end{defi}

If $f \sim \sum N^{-j} f_j$, we call $f_0$ the principal symbol of $f$, which is unique modulo $\mathcal{O} (N^{-1})$.

We next add a random perturbation to these Toeplitz operators. For this we must fix a probability space $\Omega$ with probability measure $\P$.

\begin{defi}[$\bm{\mathcal{G}_\omega}$ \textbf{and} $\bm{\mathcal{W}_\omega}$]\label{def:random}
For each $N$, let $\set{e_i : i = 1, \dots, \mathcal{N}}$ be an orthonormal basis of $H^0(X,L^N)$. Define:
\begin{align}
\mathcal{•G}_\omega = \sum _{i,j = 1}^\mathcal{N} \alpha _{j,k} e_i \otimes e_j : H^0 (X,L^N) \to H^0(X,L^N)
\end{align}
where $\alpha_{j,k}$ are independent identically distributed complex Gaussian random variables with mean zero and variance $1$. 

Similarly define $\mathcal{W}_\omega = \sum _{i,j = 1}^\mathcal{N} \tilde \alpha _{j,k} e_i \otimes e_j$, with $\tilde \alpha_{j,k}$ independent identically distributed copies of a complex random variable with mean zero and bounded second moment.
\end{defi}

The $\omega$ in the subscript of these objects is to emphasize that these objects are random. That is for each $\omega \in \Omega$, $\mathcal{ G}_\omega$ is a finite rank operator. The majority of this article describes perturbations by $\mathcal{G}_\omega$ (the Gaussian case), while a brief note at the end concerns the more general perturbations by $\mathcal{•W}_\omega$.

This paper will prove almost sure weak convergence of the empirical distribution of eigenvalues of randomly perturbed Toeplitz operators. The principal symbol of $f$ must also satisfy the property that there exists $\kappa \in (0,1]$ such that
\begin{align}
\mu_d (\{x \in X : |f_0 (x) - z | ^2 \le t \} ) = \mathcal{O}(t^\kappa)\label{eq:f_0 property}
\end{align}
as $t\to 0$ uniformly for all $z \in \C$. It is observed in \citep{Christiansen} that if $f$ is real analytic, then \eqref{eq:f_0 property} holds. See \citep{Christiansen}, and references presented there, for further discussion of \eqref{eq:f_0 property}.

\begin{theorem}[\textbf{Main Theorem}]\label{thm.main}
Given $f \in S(1)$ which satisfies \eqref{eq:f_0 property} and $\mathcal{ G}_\omega$, a family of random operators on $H^0(X,L^N)$, as defined in Definition \ref{def:random}, then for each $\e > 0$ there exists $\beta = \beta(\e) \in (0,1)$ and $C> 0$ such that if $\delta = \delta (N)$ satisfies
\begin{align}
C e ^{-N^\beta} <  \delta < C  ^{-1 }N^{-d/2 - \e} \label{eq294}
\end{align}
then we have almost sure weak convergence of the empirical measures of $T_N f + \delta \mathcal{G}_\omega$ to $\vol(X)^{-1} (f_0)_* \mu_d$.

More precisely, if $\lambda_i = \lambda_i (N,\omega)$ are the (random) eigenvalues of $T_{N}f + \delta \mathcal{G}_\omega$, then for all $\phi \in C_0^\infty (\C)$ 
\begin{align}
\frac{1}{\mathcal{ N}}\sum_{i=1}^{\mathcal{N}} \phi (\lambda_i )\xrightarrow{N\to \infty} \frac{1}{\vol(X)} \int_{\C} \phi (z) [(f_0 ) _* \mu_d](\dd z) 
\end{align}
almost surely, where $(f_0 ) _* \mu_d$ is the push-forward of the volume form $\mu_d$ on $X$ by $f _0 $.

Moreover, for each $\e>0$, the constant $\beta(\e)$ in \eqref{eq294} can be chosen at most strictly less than
\begin{align}
 \begin{cases}
 2 \epsilon \kappa & \text{if } \epsilon < \frac{1}{2(\kappa + 1)}\\
\frac{\kappa }{\kappa + 1} & \text{if } \epsilon\ge \frac{1}{2(\kappa + 1)}
\end{cases}
\end{align}
where $\kappa$ is defined in \eqref{eq:f_0 property}.
\end{theorem}

We expect Theorem \ref{thm.main} to hold for a much larger class of random perturbations than described in Definition \ref{def:random}. Indeed, the only properties of $\mathcal{ G}_\omega$ we use is a norm bound (Lemma \ref{lemma:norm gaussian matrix}) and an anti-concentration bound (Proposition \ref{thm:f1}). See \cite{Vogel2021} where Vogel and Zeitouni establish similar logarithmic determinant estimates with these classes of random perturbations, and \cite[Remark 1.3]{basak2020} where Basak, Paquette, and Zeitouni describe random perturbations satisfying these properties.

Here we present a version of Theorem \ref{thm.main} for the more general random perturbations $\mathcal{W}_\omega$ as described in Definition \ref{def:random}. 

\begin{theorem}[\textbf{General Perturbations}]\label{theorem:general perturbations}
For $\mathcal{W}_\omega$ defined in Definition \ref{def:random}, $f\in S(1)$ satisfying \eqref{eq:f_0 property}, $\delta = N^{-d}$, then the empirical measures of $T_N f + \delta \mathcal{W}_\omega$ converge almost surely to $(\vol(X) )^{-1} (f_0 )_* \mu _d $.
\end{theorem}

A proof of this result is presented in Section \ref{section:general random perturbations}.

\begin{rem}
We expect a wider range of $\delta$'s and more general random perturbations in Theorem \ref{theorem:general perturbations} should lead to the same conclusion. 
\end{rem}

\section{Review of an Exotic Calculus of Toeplitz Operators} \label{results}

In proving Theorem \ref{thm.main}, non-negative symbols are scaled by powers of $N^{-1}$. These functions belong to a more exotic symbol class than smooth functions uniformly bounded in $N$. Toeplitz operators of functions in this symbol class still have natural composition formulas. A summary of these results is contained in this section. For proofs see \citep{Izak}.

\begin{defi}[\textbf{Order Function}]
For $\rho \in [0,1/2)$, a $\rho$-order function $m$ on $X$ is a function $m\in C^\infty (X; \R_{>0})$, depending on $N$, such that there exists $M_0 \in \N$ such that for all $x,y \in X$:
\begin{align}
m(x) / m(y) \lesssim (1 + \dist (x,y) N^\rho ) ^{M_0}, \label{order function}
\end{align}
where $\dist (x,y)$ is the distance between $x$ and $y$ with respect to the Riemannian metric on $X$ induced by the symplectic form $\sigma$.
\end{defi}

\begin{defi}[$\bm{S_\rho (m)}$]\label{def.symbol class}
Given $\rho \in [0,1/2)$ and a $\rho$-order function $m$ on $X$. $S_\rho (m)$ is defined as the set of smooth functions on $X$ depending on $N$ such that for all $i\in \mathcal{I}$, $\alpha \in \N^d$:
\begin{align}
|\p^\alpha (f \circ \zeta_i ^{-1} (x))| \lesssim_\alpha N^{\delta |\rho|} m \circ \zeta_i^{-1}(x)
\end{align}
for all $x\in \zeta_i (U_i)$ (recall $\set{(U_i,\zeta_i ) : i\in \mathcal{I}}$ is a finite atlas on $X$).
\end{defi}

\begin{prop}[\textbf{Composition}]\label{thm.composition}
Given $\rho \in [0,1/2)$, $\rho$-order functions $m_1,m_2$ on $X$, $f\in S_\rho (m_1)$ and $g \in S_\rho (m_2)$. Then there exists $h \in S_\rho (m_1 m_2)$ such that:
\begin{align}
T_{N}f \circ T_{N}g = T_{N}h + \mathcal{O} (N^{-\infty}),
\end{align}
where $\mathcal{ O}$ is in terms of the norm from $L^2 (X ,L^N ) \to L^2 (X,L^N)$. Moreover, the principal symbol of $h$ is $f_0 g_0$.
\end{prop}

\begin{claim}\label{claim:poop}
Given $f\in S(1)$ with $f_0 \ge 0$, then if $\rho \in [0,1/2)$, $m(x) = f_0 N^{2\rho} + 1$ is a $\rho$-order function on $X$ and $f N^{2\rho} \in S_\rho (m)$.
\end{claim}

\begin{prop}[\textbf{Parametrix Construction}]\label{thm.parametrix construction}
Given $\rho\in [0,1/2)$, a $\rho$-order function $m$ on $X$, $\rho \in [0,1/2)$, and $f\in S_\rho (m)$ such that there exists $C>0$ so that $f> Cm$. Then there exists $g\in S_\rho (m^{-1})$ such that:
\begin{align}
T_{N} f \circ T_{N}g &=1 + \mathcal{O }(N^{-\infty}), && T_{N}g \circ T_{N} f = 1 + \mathcal{O }(N^{-\infty}).
\end{align}
\end{prop}

\begin{prop}[\textbf{Functional Calculus}]\label{thm.functional calc}
Given a $\rho$-order function $m \ge 1$ on $X$ (for a fixed $\rho \in [0,1/2)$), a family of operators $\set{R_N}_{N\in \N}$ mapping $H^0 (X,L^N)$ to itself such that $\norm{R_N}_{ } = \mathcal{O} (N^{-\infty})$ and $T_{N}f + R_N$ is self-adjoint for all $N$, and $f\in S_\rho (m)$ taking real non-negative values such that there exists $C> 0 $ with $|f| \ge m C^{-1} - C$. Then for any $\chi \in C^\infty (\R ; \C)$, there exists $g\in S_\rho (m^{-1} ) $ such that
\begin{align}
\chi(T_{N}f + R_N) = T_{N}g + \mathcal{O}(N^{-\infty})
\end{align} 
and $g$ has principal symbol $\chi(f_0 ) $. 
\end{prop}

Typically, Proposition \ref{thm.functional calc} will be applied with $R_N = 0$ for all $N$.

\begin{prop}[\textbf{Trace Formula}]\label{thm.trace}
If $m$ is a $\rho$-order function on $X$ (for fixed $\rho\in[0,1/2)$), and $f\in S_\rho (m)$, then
\begin{align}
\Tr T_{N}f & = \left ( \frac{N}{2\pi} \right )^d\int _X f(x) \dd \mu_d(x)+ \mathcal{O} (N^{d - (1-2\rho)} ) \max_{x \in X} m(x) \\
&= \left ( \frac{N}{ 2\pi} \right )^d\int _X f_0(x) \dd \mu_d(x)+ \mathcal{O} (N^{d - (1-2\rho)} ) \max_{x \in X} m(x),
\end{align}
where $f_0$ is the principal symbol of $f$.
\end{prop}

Note that if $f = 1$, then $\Tr T_{N}1 = \Tr (\Pi_N ) = \dim(H^0(X,L^N)) = \mathcal{N}$ which is an alternative way of proving that $\mathcal{N} = \vol(X)(N/2\pi)^d + \mathcal{O} (N^{d-1})$.

\section{Probabilistic Preliminaries} \label{log potential}
This paper uses the probabilistic machinery of logarithmic potentials. A brief overview is presented in this section.

\begin{defi}[$\bm{\mathcal{ P}(\C)}$]
Let $\mathcal{ P} (\C)$ be the collection of probability measures $\mu $ on $\C$ such that $\int \log (1 + |z| ) \dd \mu(z) < \infty$.
\end{defi}

\begin{defi}[\textbf{Logarithmic Potential}]
For $\nu \in \mathcal{P} (\C)$, define the logarithmic potential as: $U_\nu(z) : = \int _\C \log | z - w | \dd \nu (w)$.
\end{defi}

Using the fact that $\log|z|$ is the fundamental solution of the Laplacian, it can be shown that, in the sense of distributions, $\Delta U_\nu = 2\pi \nu$, which is the key ingredient in proving the following theorem.

\begin{prop}[\textbf{Convergence of Random Measures by Logarithmic Potentials}]\label{thm:log_potential}
Given $\set{\nu_N}\subset \mathcal{P}(\C)$ random measures such that almost surely $\supp \nu_N \subset \Lambda$ for $N \gg 1$ (with $\Lambda \Subset \bar \Lambda \Subset \Lambda' \Subset \C$) and for almost all $z\in \Lambda'$: $U_{\nu_N } (z) \to U_\nu (z)$ almost surely for some $\nu \in \mathcal{P} (C)$ with $\supp \nu \subset \Lambda$. Then almost surely $\nu_N \to \nu$ weakly.
\end{prop}

\begin{proof}
See \citep[Theorem 7.1]{Sjo-vogel}.
\end{proof}

We wish to use Proposition \ref{thm:log_potential} to prove almost sure weak convergence of the empirical measures of $T_Nf + \delta \mathcal{G}_\omega$.

\begin{defi}[$\bm{\nu_N}$]
Let $\sigma_N$ be the spectrum of $T_{N}f + \delta \mathcal{G}_\omega$. Let $\nu_N = \mathcal{N}^{-1} \sum _{\lambda \in \sigma_N} \hat \delta _\lambda$ where $\delta >0$ depends on $N$, and $\hat \delta_\lambda$ is the Dirac distribution centered at $\lambda$. The logarithmic potentials for these random measures are
\begin{align}
U_{\nu_N} (z) = \frac{1}{\mathcal{N}} \sum _{\lambda \in \sigma_N} \log |z - \lambda| = \frac{1}{\mathcal{N}} \log |\det (T_{N}f + \delta \mathcal{G}_\omega - z) |.
\end{align}
\end{defi}

\begin{defi}[$\bm{\nu}$]
Let $\nu = \vol(X)^{-1} (f_0 ) _* \mu_d$ (recall $\mu_d$ is the volume measure on $X$) which has logarithmic potential
\begin{align}
U_\nu (z) = \strokedint _{X} \log |z - f_0 (x ) | \dd \mu_d(x).
\end{align}
\end{defi}
Where $\strokedint_X f \dd \mu_d$ is defined as $\vol(X)^{-1} \int f\dd \mu_d $. 

\begin{claim} 
For all $N$, $\nu_N,\nu\in \mathcal{P}(\C)$.
\end{claim}
\begin{proof}
For each $N\in \N$
\begin{align}
\int _\C \log (1 + |z| ) \dd \nu_N(z) & = \frac{1}{\mathcal{ N}} \sum _{\lambda \in \sigma_N} \log (1 + |\lambda| ) \\
&\le \max _{\lambda \in \sigma _N } \log (1 + |\lambda| ) \\
&\le \log (1 + \norm{T_{N}f + \delta \mathcal{G}_\omega}_{} ) < \infty.
\end{align}
And similarly,
\begin{align}
\int_\C \log (1 + |z| ) \dd \nu(z) &= \frac{1}{\vol(X)} \int _\C \log (1+ |z| )[ (f _0)_* \mu_d](\dd z) \\
&\le \max_{x\in X} \log (1 + |f(x)| ) < \infty.
\end{align}
\end{proof}

Let $\Lambda$ be a neighborhood of $f(X)$. Clearly $\supp \nu \subset \Lambda$, the same is true with probability $1$ for $\nu_N$, for sufficiently large $N$. A standard random matrix lemma is required to show this.

\begin{lemma}[\textbf{Norm of Gaussian Matrix}]\label{lemma:norm gaussian matrix}
There exists $C >0$ such that
\begin{align}
\P(\norm{\mathcal{G}_\omega}_{  } \le C \mathcal{N}^{1/2} )\ge 1 - \exp(- \mathcal{N}).
\end{align}
If an event has this lower bound of probability, it is said to occur with overwhelming probability.
\end{lemma}
\begin{proof}
See \citep[Exercise 2.3.3]{Tao}.
\end{proof}
For a fixed $\e> 0$, we will choose $\delta = \delta (N)$ such that 
\begin{align}
0 < \delta = \mathcal{O} (\mathcal{N}^{-1/2 - \e}). \label{eq:choice of delta}
\end{align}

\begin{lemma}[\textbf{Borel--Cantelli}]\label{lemma:Borel-Cantelli}
If $A_n$ are events such that $\sum_1^\infty \P (A_n) < \infty$, then the probability that $A_n$ occurs infinitely often is $0$. 
\end{lemma}
\begin{proof}
See \citep{Durrett}.
\end{proof}
\begin{lemma}[\textbf{Bound of $\bm{T_{N}f}$}]\label{lemma:crude TNf bound}
Given $f\in S(1)$, then $\norm{T_{N}f}_{L^N \to L^N} \le \sup |f|$.
\end{lemma}
\begin{proof}
This follows immediately by writing $T_N f = \Pi_N \circ M_f \circ \Pi_N$ and recalling that $\Pi_N$ is unitary.
\end{proof}

\begin{claim}\label{claim:4.2}
Almost surely, $\supp \nu_N \subset \Lambda$ for $N \gg 1$.
\end{claim}
\begin{proof}
First note that $\norm{T_N f + \delta \mathcal{G} _\omega}_{ } \le \norm{T_N f}_ { } + \delta \norm{\mathcal{ G}_\omega}_{ } \le \sup f + \mathcal{N}^{-\e}$ with overwhelming probability (by Lemma \ref{lemma:norm gaussian matrix}, \eqref{eq:choice of delta}, and Lemma \ref{lemma:crude TNf bound}). Let $\sigma_N$ be the spectrum of $T_N f + \delta \mathcal{G} _\omega$. In this event, for sufficiently large $N$, $\sigma_N \subset \Lambda$. So if $A_N^c$ is the event that $\sigma _N \subset \Lambda$, then $\P (A_N^c ) \ge 1 - e^{-\mathcal{N}}$. Therefore $\sum \P (A_N) < \infty$ and so by Lemma \ref{lemma:Borel-Cantelli}, almost surely $P(A_N^c ) = 1$ for $N \gg 1$.

\end{proof}

\begin{lemma}[\textbf{Almost Sure Convergence}]\label{lemma:ASC}
If $\set{Y_N} _{N\in \N}$ and $Y$ are random variables on a probability space $(\Omega, \P)$ and $\e_N$ is a sequence of numbers converging to $0$ such that
\begin{align}
\sum _{N=1}^\infty \P(|Y_N - Y| > \e_N ) < \infty,
\end{align}
then $Y_N \to Y$ almost surely.
\end{lemma}
\begin{proof}
See \citep{Durrett}.
\end{proof}

Therefore $\nu_N$ and $\nu$ satisfy the conditions of Proposition \ref{thm:log_potential}. So it suffices to show that $U_{\nu_N} (z) \to U_\nu (z)$ for almost all $z$ in the bounded set containing $\Lambda$. To prove this almost sure convergence, it suffices to apply Lemma \ref{lemma:ASC} with $Y_N= \mathcal{N}^{-1} \log |\det (T_{N}f+ \delta \mathcal{G}_\omega - z) |$ and $Y = \strokedint \log |z - f_0 (x ) | \dd \mu_d(x)$ for suitably chosen $\e_N$.

\section{Setting up a Grushin Problem} \label{grushin approach}
To control $\log |\det (T_N f + \delta \mathcal{G}_\omega - z)|$ we follow the now standard method of setting up a Grushin problem. This approach was used in \citep{Vogel} and \citep{Hager}, and is comprehensively reviewed in \citep{SjoZwo}.

Let $P = T_{N}f$ and $\mathcal{H}_N = H^0(X,L^N)$. Define the $z$-dependent self-adjoint operators $Q = (P-z)^* (P-z)$ and $\tilde Q = (P-z)(P-z)^*$. These operators share the same eigenvalues $0 \le t_1^2 \le \cdots \le t_{\mathcal{ N}}^2$. We can find an orthonormal basis of eigenvectors of $Q$ for these eigenvalues, denoted by $e_i$, and similarly, and orthonormal basis of eigenvectors of $\tilde Q$ denoted by $f_i$. These eigenvectors can be chosen such that
\begin{align}
(P-z)^* f_i = t_i e_i, && (P-z)e_i = t_i f_i, && i = 1, \dots, \ \mathcal{N}.
\end{align}
Next we fix $\rho \in (0,\min(1/2, \e))$, and define:
\begin{align}
\alpha := N^{-2\rho}, && A: = \max \set{i\in \Z : t_i^2 \le \alpha}. \label{eq:alpha definition}
\end{align}

\begin{defi}[$\bm{\mathcal{ P}^\delta}$]
Let $\delta_j$ be the standard basis of $\C^A$, and define the operators $R_+(z) =\sum _1^A \delta_i \otimes e_i : \mathcal{H}_N \to \C^A $ and $R_- (z)= \sum_1^A f_i \otimes \delta _i : \C^A \to \mathcal{H}_N$, where we use	 the notation $(u\otimes v ) (w) = \ip{w}{v}u$. For each $z\in \C$ and $\delta \ge 0$, define
\begin{align}
\mathcal{P}^\delta(z) := \mat{P + \delta \mathcal{G}_\omega - z & R_-(z) \\ R_+(z) & 0 } : \mat{ \mathcal{H}_N \\ \C^A} \to \mat{ \mathcal{H}_N \\ \C^A} .\label{eq:scriptP}
\end{align}
\end{defi}

\begin{lemma}
If $\delta = 0$, then $\mathcal{P}^\delta$, as defined in \eqref{eq:scriptP}, is bijective with inverse \begin{align}
\mathcal{E} ^0 (z) = \mat{\sum _{A +1}^{\mathcal{ N}} \frac{1}{t_i} e_i \otimes f_i && \sum_1^A e_i \otimes \delta _i\\
\sum_1^A \delta _i \otimes f_i && - \sum_1^A t_i \delta _i \otimes \delta _i } : = \mat{ E^0(z) & E_+^0(z) \\ E^0_- (z) & E_{-+}^0(z)}. \label{eq:scriptE}
\end{align}

\end{lemma}
\begin{proof}
See \citep[Section 5.1]{Vogel}.
\end{proof}

To ease notation, the $z$ in the argument for these operators will often be dropped. Unless specified, all estimates are uniform in $z$.

\begin{claim}[\textbf{Invertibility of $\bm{\mathcal{ P}^\delta$}}]\label{claim:5.1}
$\mathcal{P}^\delta $ is invertible if $\delta \norm{\mathcal{ G}_\omega E^0 }\ll 1$.
\end{claim}
\begin{proof}
By computation	
\begin{align}
\mathcal{P}^\delta \mathcal{E}^0 = 1 + \mat{\delta \mathcal{G}_\omega E^0 & \delta \mathcal{G}_\omega E_+^0 \\ 0 & 0 } : = 1 + K.
\end{align}
If $\norm{K} < 1$ (which is true given the hypothesis), then $(I+K)^{-1}$ exists as a Neumann series, and we get $\mathcal{ P}^\delta \mathcal{E}^0 (I+K)^{-1} = I$ (a similar argument shows this is a left inverse as well).
\end{proof}

\begin{lemma}[\textbf{Norm of $\bm{E^0}$}]\label{lemma.norm.E}
In the notation of \eqref{eq:scriptE}, $\norm{E^0}_{ } \le \alpha ^{-1/2}$.
\end{lemma}
\begin{proof}
By construction, $E^0 = \sum_{M+1}^{\mathcal{N}} (t_i)^{-1} e_i \otimes f_i$, so that $\norm{E^0}_{ } = \norm{E^0 f_{M+1}}_{ } = (t_{M+1})^{-1} \le \alpha ^{-1/2}$.
\end{proof}

\begin{lemma}[\textbf{Norm of $\bm E_+^0 $}]
In the notation of \eqref{eq:scriptE}, $\norm{E_+^0}_{ } = 1$.
\end{lemma}
\begin{proof}
By construction $E_+^0(z) = \sum_1^M e_i \otimes \delta_i$ which has norm 1.
\end{proof}

These lemmas, along with Lemma \ref{lemma:norm gaussian matrix}, guarantee that if $\delta = \mathcal{O}( \alpha ^{1/2} \mathcal{N}^{-1/2})$, then $\mathcal{ P}^\delta$ is invertible with overwhelming probability. Denote the inverse of $\mathcal{ P}^\delta$ by $\mathcal{ E}^\delta$ with the same notation for its components as in \eqref{eq:scriptE}.

Define $P^\delta = P + \delta \mathcal{G}_\omega$. By Schur's complement formula, if $P^\delta -z$ is invertible,
\begin{align}
\det \mat{P^\delta -z & R_- \\ R_+ & 0 } = \det(P^\delta - z) \det (-R_+ (P^\delta - z)^{-1} R_-).
\end{align}
Writing $\mathcal{ P}^\delta \mathcal{E}^\delta = 1$, we get that $-R_- = (P^\delta -z) E_+^\delta (E_{-+}^\delta)^{-1}$ and $R_+ E_+^\delta = 1$. Therefore $-R_+ (P^\delta - z)^{-1} R_- = (E_{-+}^\delta )^{-1}$, so that
\begin{align}
\log |\det (P^\delta - z ) | = \log | \det \mathcal{P }^\delta (z) | + \log |\det E_{-+}^\delta (z) | . \label{Schur}
\end{align}
Note that $P^\delta - z$ is invertible if and only if $E_{-+}^\delta$ is invertible. Therefore \eqref{Schur} holds even when $P^\delta - z$ is not invertible.

Therefore, to prove Theorem \ref{thm.main}, it suffices to show summability of the probability of the events:
\begin{align}
\mathcal{A}_N : = \set{ \left |\underbrace{ (\mathcal{N})^{-1} (\log | \det \mathcal{P}^\delta | + \log |\det E_{-+}^\delta (z) | ) - \fint_X \log |z-f_0(x) | \dd \mu }_{: = B} \right | > \e_N}.
\end{align}
We let $\e_N = N^{-\gamma}$ for a suitably chosen $\gamma =\gamma (d,\kappa)> 0$. Expand $B = B_1 + B_2 + B_3 $ where:
 \begin{align}
 B_1 &= \mathcal{N}^{-1} \log |\det \mathcal{P} ^0| - \strokedint_X \log |z - f_0 (x) | \dd \mu(x), \label{eq:B_1}\\
 B_2 &=\mathcal{N}^{-1}(\log |\det \mathcal{P}^\delta | - \log |\det \mathcal{P}^0| ) , \label{eq:B_2}\\
 B_3 & = \mathcal{N}^{-1} \log | \det E_{-+}^\delta | \label{eq:B_3} .
 \end{align}

Controlling $B_1$ requires the most work as it requires utilizing the calculus of Toeplitz operators. However, it is completely deterministic, and remains true for unperturbed operators. $B_2$ will be easily shown to be negligible. Proving a lower bound on $B_3$ is the key ingredient in proving Theorem \ref{thm.main}, as it will force the events $\mathcal{A}_N$ to sufficiently small probability. Without a perturbation, $B_3$ will have no lower bound. 

Proving bounds on $B_2$ and $B_3$ closely follow \citep{Vogel}.

\begin{lemma}[\textbf{Bound on $\bm{E_{-+}}$}]\label{lemma:bound4}
In the notation of \eqref{eq:scriptE}, $\norm{E_{-+}^0}_{ } \le \sqrt{\alpha}$.
\end{lemma}
\begin{proof}
By construction, $E_{-+}^0 = -\sum_1^A t_j \delta_j \otimes \delta_j$, so $\norm{E_{-+}^0}_{ } = |E_{-+}^0(\delta_A)| = t_A\le \sqrt{\alpha}$.
\end{proof}

\begin{lemma}[\textbf{Bound on $\bm{E^\delta}$}]\label{lemma:bound E tau}
In the notation of \eqref{eq:scriptE}, $\norm{E^\delta}_{ } \le 2 \alpha ^{-1/2}$ with overwhelming probability.
\end{lemma}
\begin{proof}
By the Neumann construction, $\norm{E^\delta}_{ } = \norm{E^0(1 + \delta \mathcal{G}_\omega E^0)^{-1}}_{ } \le 2 \norm{E^0}_{ }$ which is bounded by $2 \alpha ^{-1/2} $ by Lemma \ref{lemma.norm.E}.
\end{proof}

\begin{claim}[\textbf{Bound on $\bm{B_2}$}]\label{claim.b2 bound}
In the notation of \eqref{eq:B_2}, $ B_2 = \mathcal{O}(\delta \alpha^{-1/2} \mathcal{N}^{1/2} )$ with overwhelming probability.
\end{claim}
\begin{proof}
Using Jacobi's formula, $(\log \det A )' = \Tr(A^{-1} A')$, we have that
\begin{align}
\mathcal{N}B_2 &= \log |\det \mathcal{P}^\delta | - \log |\det \mathcal{P} | = \int_0^\delta \frac{d}{d\tau} \log |\det \mathcal{P} ^\tau| \dd \tau \\
&= \int _0^\delta \Re{ \Tr ( \mathcal{E}^\tau \frac{d}{d \tau} \mathcal{P}^\tau ) } \dd \tau = \int_0^\delta \Re{ \Tr (E^\tau \mathcal{G}_\omega)} \dd \tau .
\end{align}
 Taking absolute values and using properties of trace norms
\begin{align}
|\log |\det \mathcal{P}^\delta | - \log |\det \mathcal{P}^0 || &\le \delta \sup _{\tau \in [0,\delta]} \norm{E^\tau}_{ } \norm{\mathcal{G}_\omega}_{tr} \le \mathcal{O} (\delta \alpha^{-1/2} \mathcal{N} \norm{ \mathcal{G}_\omega}) , \label{eq.3} 
\end{align}
where we used Lemma \ref{lemma:bound E tau}, and H\"older's inequality for the Schatten norm. Recalling the bound on $\mathcal{ G}_\omega$, \eqref{eq.3} is $\mathcal{O} (\delta \alpha^{-1/2} \mathcal{N}^{3/2})$ with overwhelming probability. 
\end{proof}

The following theorem about singular values of randomly perturbed matrices is required for proving a lower bound of $B_3$. Given a matrix $B$, let $s_1(B) \ge s_2(B) \ge \cdots \ge s_N(B)$ be its singular values.
\begin{prop}\label{thm:f1}
If $B$ is an $N \times N$ complex matrix and $\mathcal{G}_\omega$ is a random matrix with independent identically distributed complex Gaussian entries of mean $0$ and variance $1$, then there exists $C> 0$ such that for all $\delta >0$, $t > 0$:
\begin{align}
\mathbb{P}(s_N(B + \delta \mathcal{G}_\omega) < \delta t ) \le C Nt^2 .
\end{align}
\end{prop}
\begin{proof}
See \citep[Theorem 23]{Vogel}, which is a complex version proven by Sankar, Spielmann and Teng in \citep[Lemma 3.2]{Sankar}.
\end{proof}

\begin{claim}[\textbf{Bound on $\bm{B_3}$}]\label{claim_bound b_3}
In the notation of \eqref{eq:B_3}, $B_3$ obeys the probabilistic upper bound
\begin{align}
\P( \mathcal{N}^{-1} \log |\det E_{-+}^\delta | < 0) > 1 - e^{-\mathcal{N}}, \label{eq:claimUp}
\end{align}
for $N \gg 1$. And $B_3$ obeys the probabilistic lower bound: there exists there exists $C> 0 $ such that for all $\delta > 0$
\begin{align}
\P\left (\mathcal{N}^{-1} \log |\det E_{-+}^\delta | \ge A \mathcal{N}^{-1} \log (\delta t) \right )> 1 - C \mathcal{N} t^2 - e^{-\mathcal{N}}.
\end{align}
\end{claim}
\begin{proof}
First, by the Neumann series construction and choice of $\delta$, with overwhelming probability,
\begin{align}
\norm{E_{-+}^\delta}_{ } \le \norm{E_{-+}^\delta - E _{-+}^0}_{ } + \norm{E_{-+}^0}_{ } &= \norm{ E_-^0 (1 - \delta \mathcal{G}_\omega E^0)^{-1} \delta \mathcal{G}_\omega E_+^0}_{ } + \norm{E_{-+}^0}_{ } \\
& \le 2 \norm{\delta \mathcal{G}_\omega}_{ } +\alpha ^{1/2} \le C \alpha^{1/2}. \label{eq:725}
\end{align}

So, in this event, $\norm{E_{-+}^\delta}_{ } \le C \alpha ^{1/2} < 1$ for $N \gg 1$, and therefore $\log |\det E_{-+}^\delta |< 0 $ proving \eqref{eq:claimUp}.

For the lower bound, first note that
\begin{align}
\log |\det E_{-+}^\delta | = \sum_1^A \log s_j (E_{-+}^\delta ) \ge A \log s_A(E_{-+}^\delta ) .
\end{align}
For a matrix $B$, let $t_1(B)$ be the smallest eigenvalue of $\sqrt{B^*B}$, so $s_A (E_{-+}^\delta ) = t_1 (E_{-+}^\delta)$. Assume that $P-z$ is invertible. Using that $(E_{-+}^0)^{-1} = - R_+ (P-z)^{-1} R_- $ and properties of singular values of sums and products of trace class operators, we get
\begin{align}
(t_1(E_{-+}^0))^{-1} &= s_1((E_{-+}^0)^{-1}) \le s_1(R_-) s_1 (R_+ ) s_1 ((P-z)^{-1}) = \norm{R_+}_{ } \norm{R_-}_{ } s_1 ((P-z)^{-1})\\
& = s_1 ((P-z)^{-1}) =( t_1(P-z))^{-1} = s_{\mathcal{N}}( (P-z)^{-1}).
\end{align}
For $\delta = \mathcal{O}( \mathcal{N}^{-1/2} \alpha^{1/2})$, this holds for $E_{-+}^\delta$ (the event of a singular matrix has probability zero and the singular values depend continuously on $\delta$) so $s_A (E_{-+}^\delta ) = t_1(E_{-+}^\delta ) \ge s_{\mathcal{N}} (P + \delta \mathcal{G}_\omega-z)$ with overwhelming probability.

Using Proposition \ref{thm:f1}, in the event that $\norm{ \mathcal{G}_\omega}_{ } \le C\mathcal{N} ^{1/2}$ (overwhelming probability) and $s_{ \mathcal{N}}(P-z +\delta \mathcal{G}_\omega) > \delta t $ (probability at least $1 - C \mathcal{N} t^2$), we have that $s_A(E_{-+}^\delta ) > \delta t$ with probability greater than $1- C \mathcal{N} t^2 - e^{-\mathcal{N}}$. Therefore
\begin{align}
\log |\det E_{-+}^\delta | \ge A \log s_A (E_{-+}^\delta ) \ge A \log (\delta t)
\end{align}
with probability $\ge 1 - e^{- \mathcal{N}} - C\mathcal{N}t^2$.
\end{proof}

\section{Bound on $B_1$} \label{section b1}
This section is devoted to estimating $B_1$ (as in \eqref{eq:B_1}) which involves computing the trace of a function of a Toeplitz operator belonging to an exotic symbol class. This closely follows \citep{Vogel}, however several simplifications arise partially due to requiring weaker bounds, and several modifications are required as we are working with Toeplitz operators.
\begin{claim}[\textbf{Bound on $\bm{B_1}$}]\label{claim.bound B_1}
For $\mathcal{ P}$ defined in \eqref{eq:scriptP},
 \begin{align}
 \log |\det \mathcal{P}^0 | = N^d \strokedint _{X} \log |f_0 (x) - z |^2 \dd\mu +\mathcal{O} ( N^{d-\min(2\rho\kappa,(1-2\rho)) } \log (N) ).
 \end{align}
\end{claim}

\begin{proof}

Let's first consider some preliminary reductions in computing $\log |\det \mathcal{P} ^0| $. By Schur's complement formula, $|\det \mathcal{P} ^0| ^{2} = |\det (P-z) |^2 |\det E_{-+}^0|^{-2}$. The first term is:
\begin{align}
|\det (P-z)|^2 = \det Q = \prod_{i=1}^{\mathcal{N}} t_i^2.
\end{align}
Because $E_{-+}^0 = - \sum_1^{A} t_j \delta_j \otimes \delta_j$ (recall $A$ is the largest integer such that $t^2_A \le \alpha$), the second term is
\begin{align}
|\det E_{-+}^0|^{-2} = \left( \prod _{i=1}^{A} t_i^2 \right )^{-2},
\end{align}
therefore 
\begin{align}
|\det \mathcal{P}^0|^2 = \prod_{i= A+1}^{\mathcal{N}} t_i^2 = \alpha ^{-A} \prod_{i=1}^{\mathcal{N}} 1_\alpha (t_i ^2) =\alpha ^{-A } \det 1_\alpha (Q)
\end{align}
where $1_\alpha = \max (x,\alpha)$. If $\chi$ is a cut-off function identically $1$ on $[0,1]$, and supported in $[-1/2,2]$, then $x + (\alpha /4) \chi (4x / \alpha) \le 1_\alpha (x) \le x + \alpha \chi (x / \alpha) $ for $x \ge 0$. Therefore
\begin{align}
\det \left ( Q + 4^{-1} \alpha \chi \left ( Q/(4^{-1} \alpha)\right ) \right ) \le \det (1_\alpha (Q) )\le \det \left ( Q + \alpha \chi (Q/\alpha ) \right ). \label{eq:funny bound}
\end{align}

Now fix $1 \gg \alpha_1 > \alpha$, so that $\log \det (Q + \alpha \chi (Q/\alpha ) )$ can be written
\begin{align}
-\int_\alpha^{\alpha_1} \frac{d}{dt} \log \det (Q + t \chi (Q/t)) \dd t + \log \det (Q+\alpha _1 \chi (Q/\alpha _1 )). \label{eq pg9}
\end{align}
First the integrand is estimated. Let $\psi(t) = (t - t \chi ' (t)) (1 + \chi (t) )^{-1}$ so that
\begin{align}
\frac{d}{dt} \log (x + t \chi (x/t)) = t^{-1} \psi (x/t)
\end{align}
for $t > 0 $ and $\psi \in C_0^\infty (\R_{\ge 0})$. Therefore, by Jacobi's identity,
\begin{align}
\frac{d}{dt} \log \det (Q + t \chi (Q/t)) = \Tr (t ^{-1} \psi (Q/t) ) .
\end{align}
While morally the same, here we diverge from \citep{Vogel}'s proof to handle this trace term, and must rely on Section \ref{results}. The main issues are that $Q$ is the composition of Toeplitz operators, which may no longer be a Toeplitz operator (but is modulo $\mathcal{ O}(N^{-\infty})$ error), $Q/t$ belongs to an exotic symbol class so to compute $\psi(Q/t)$ requires an exotic calculus, and the trace formula (Proposition \ref{thm.trace}) has weaker remainder than for quantizations of tori.

Let $\rho_t $ be such that $t = N^{-2 \rho _t}$. By Proposition \ref{thm.composition}, $Q = T_{N}q + \mathcal{O} (N^{-\infty})$, where the principal symbol of $q$ is $|f_0 - z|^2$. For each $t$, $Q/t$ is (modulo $\mathcal{ O} (N^{-\infty})$) a Toeplitz operator with symbol in $S_{\rho_t} (m_t)$ where $m_t = q_0 / t + 1$, by Claim \ref{claim:poop}. And so, by Proposition \ref{thm.functional calc}, there exists $q_t \in S_{\rho_t} (m_t^{-1} )$, such that $\psi (Q/t) = T_N (q_t) + E_N (t)$. Where $q_t$ has principal symbol $\psi(q/t)$ and $E_N (t) = \mathcal{O} (N^{-\infty})$ (with estimates uniform over $t$). Therefore
\begin{align}
\int_\alpha^{\alpha_1} \frac{d}{dt} \log \det (Q+ t \chi (Q/t) ) \dd t &= \int_\alpha^{\alpha_1} \Tr (t^{-1} \psi (Q/t) ) \dd t \\
&= \int_\alpha^{\alpha_1} t^{-1} \Tr ( T_{N} (q_t) + E_N (t) ) \dd t .
\end{align}
The error term is
\begin{align}
\int_\alpha^{\alpha_1} t^{-1} \Tr (E_N(t)) dt = \mathcal{O} (N^{-\infty})
\end{align}
because $E_N(t)$ is uniformly $\mathcal{ O}(N^{-\infty})$. While for each $t$, Proposition \ref{thm.trace} shows that
\begin{align}
\Tr (T_N (q_t) ) = \left ( \frac{N}{2\pi}\right )^d \int_X \psi(q_0 /t) \dd\mu_d (x) + t^{-1} \mathcal{O} (N^{d-1})
\end{align}
because $m^{-1}$ is bounded. Therefore
\begin{align}
\int_\alpha^{\alpha_1} \frac{d}{dt} \log \det (Q+ t \chi (Q/t) ) \dd t &= \int_\alpha^{\alpha_1} \left ( \int_X \left ( \frac{N}{2\pi}\right )^d t^{-1} \psi(q_0 /t) \dd \mu_d (x) + t^{-2} \mathcal{O} (N^{d-1}) \right ) \dd t\\
&= \left ( \frac{N}{2\pi} \right )^d \int_X \log ( q_0 + t \chi(q_0 /t) ) \Big |_{t = \alpha}^{t = \alpha_1} \dd \mu(x) + \mathcal{O}(N^{d-1} \alpha ).
\end{align}

Next the second term of \eqref{eq pg9} is computed. Because $\alpha_1$ is fixed, $Q/\alpha_1$ has symbol in $S(1)$. Therefore, by Proposition \ref{thm.functional calc}, $Q + \alpha _1 \chi(Q/\alpha_1 ) = T_{N}r + E_N$ (with $\norm{E_N}_{ } = \mathcal{O} (N^{-\infty})$) where $r \in S(1)$ with principal symbol $q_0 + \alpha_1 \chi (q_0 /\alpha_1)$. Let $r^t = t r + (1 - t) \in S(1)$, so that
\begin{align}
\log \det (Q +\alpha _1 \chi (Q/\alpha_1 ) ) &= \int_0^1 \frac{d}{dt}\log \det (T_{N}r^t + tE_N ) \dd t \\
&= \int_0^1 \Tr \left ( \left (T_{N}r^t + t E_N \right )^{-1} \left ( \frac{d}{dt} T_{N}r^t + E_N \right ) \right )\dd t.
\end{align}
The principal symbol of $r^t$ is $r_0^1 = t(q_0 + \alpha_1 \chi(q_0 / \alpha_1) ) + (1-t)$. Note that when $x\ge 0$, then $x + \alpha_1 \chi (x/\alpha_1) \ge \alpha_1 > 0$. Therefore $(r_0^t) \ge \alpha_1$ . 

\begin{lemma}\label{lemma:721}
There exists $s (t) \in S(1)$ (with bounds uniform in $t$) such that $(T_{N}r^t + t E_N)^{-1} = T_{N}s(t) + \mathcal{O} (N^{-\infty})$, and the principal symbol of $s(t)$ is $(r^t_0)^{-1}$.
\end{lemma}
\begin{proof}
By Proposition \ref{thm.parametrix construction}, there exists a symbol $\ell = \ell(t) \in S(1)$ which inverts (modulo $\mathcal{ O} (N^{-\infty})$ error) $T_N r^t$, and has principal symbol $(r_0^t)^{-1}$. But then
\begin{align}
(T_N r^t + tE_N) T_{N} \ell = 1 + K
\end{align}
with $K = \mathcal{O}(N^{-\infty})$, using that $t E_N = \mathcal{O} (N^{-\infty})$ and $T_N \ell $ has norm bounded independent of $N$. By Neumann series, for $N \gg 1$, $(1 + K)$ is invertible, so that:
\begin{align}
( T_N r^t + t E_N) (T_N \ell ) (1 + K)^{-1} = 1.
\end{align}
$(T_N \ell ) (1 + K)^{-1}$ will be a Toeplitz operator, modulo a $\mathcal{ O} (N^{-\infty})$ term, with symbol $\ell$ which has principal symbol $(r_0^t)^{-1}$. By repeating this argument, but left-composing by $T_N \ell$, we get the lemma.
\end{proof}
Clearly $\frac{d}{dt} T_{N}r^t = T_{N}(r - 1)$ so using Lemma \ref{lemma:721}, we get that
\begin{align}
\left ( T_{N}r^t + t E_N \right )^{-1} \left ( \frac{d}{dt }T_{N}r^t + E_N \right ) 
\end{align}
is (modulo $\mathcal{ O} (N^{-\infty})$) a Toeplitz operator with principal symbol $(r_0^t )^{-1} (\frac{d}{dt} r_0^t)$. So by Proposition \ref{thm.trace}
\begin{align}
\Tr \left ( \left ( T_{N}r^t + t E_N \right )^{-1} \left ( \frac{d}{dt }T_{N}r^t + E_N \right ) \right ) = \left ( \frac{N}{2\pi}\right )^d \int_X ( r_0^t)^{-1}\left ( \frac{d}{dt} r_0^t \right ) \dd \mu_d (x) + \mathcal{O} (N^{d-1})
\end{align}
which when integrated from $t = 0$ to $t = 1$ becomes:
\begin{align}
\left ( \frac{N}{2\pi} \right )^d \int_X \log (r_0^1 )dx + \mathcal{O} (N^{d-1}) = \left ( \frac{N}{ 2\pi}\right )^d \int_X \log ( q_0 + \alpha_1 \chi (q_0 / \alpha_1) )\dd \mu_d( x) + \mathcal{O} (N^{d-1}).
\end{align}

Therefore \eqref{eq pg9} becomes:
\begin{align}
\left (\frac{N}{2\pi}\right )^d \int_X \log (q_0 + \alpha \chi (q_0 / \alpha) ) \dd \mu _d+ \mathcal{O} (N^{d-1} \alpha^{-1}).
\end{align}

A calculus lemma is required to estimate $\int_X \log (q_0 + \alpha \chi(q_0 /\alpha ) ) \dd x$.
\begin{lemma}
Given $q\in C^\infty (X; \R_{\ge 0 } )$ such that $\mu_d \left ( \set{x \in X : q(x) \le t}  \right ) = \mathcal{O} (t^\kappa)$ as $t\to 0$ for $\kappa \in (0,1]$, and $\chi \in C_0^\infty((-1/2,2);[0,1])$ identically $1$ on $[0,1]$. Then
\begin{align}
\int _X \log (q + \alpha \chi (q /\alpha ) ) \dd \mu_d = \int _X \log (q)\dd \mu_d + \mathcal{O }(\alpha^\kappa ).
\end{align}
\end{lemma}
\begin{proof}
Let $g(t) = \log (t + \alpha \chi (t/\alpha))$ and $m(t) = \mu_d (\set{x \in X : q(x) \le t})$. Then, letting $q_1 = \max q + 2 \alpha$,
\begin{align}
 \int_X \log (q + \alpha \chi (q/\alpha )) - \log (\alpha) \dd \mu_d &= \int_X g (q(x)) - g(0) \dd \mu_d= \int_X \int_0^{q(x)} g'(t) \dd t \dd \mu_d\\
 &= \int_0^{q_1} g'(t) \int_{q(x) > t} \dd \mu_d \dd t= \int_0^{q_1} g'(t) (\vol(X) -m(t)) \dd t\\
 &= \vol(X)( g(q_1) - \log(\alpha) ) -\int_0^{q_1} g'(t) m(t) \dd t.
\end{align}
So that:
\begin{align}
\int_X \log (q+ \alpha \chi(q/\alpha ) \dd \mu_d = \vol (X) g(q_1) - \int_0^{q_1} g' (t) m(t) \dd t. \label{Vogel-error}
\end{align}
Similarly, if $\tilde g(t) = \log (t)$, we get an analogous expression as \eqref{Vogel-error}, that is:
\begin{align}
\int_X \log (q ) \dd \mu_d = \vol (X) \tilde g(q_1) - \int _0^{q_1} \tilde g'(t) m (t) \dd t.
\end{align}
Note that $g(q_1)= \tilde g(q_1)$. Therefore:
\begin{align}
 \left | \int_X \log (q + \alpha \chi (q /\alpha) ) - \log (q) \dd \mu_d \right | &= \left | \int_0^{q_1} (\tilde g'(t) - g'(t)) m(t) \dd t \right | \\
 &= \left | \int_0^{q_1} \left (\frac{1}{t} - \frac{1+ \chi'(t/\alpha))}{t + \alpha \chi(t/\alpha)} \right) m(t) \dd t \right |\\
 &= \left | \int_0^{q_1/\alpha} \left ( \frac{1}{s} - \frac{1+ \chi' (s)}{s+ \chi(s)} \right ) m(s \alpha )\dd s \right | \\
 &\lesssim \int_0^{2} s^{-1} m(s\alpha ) \dd s \\
 &\lesssim \alpha^\kappa \int_0^{2} s^{\kappa - 1} \dd s \lesssim \alpha^{\kappa}.
\end{align}
Here we use that $\chi(0) = 1$ to get a lower bound on $|s + \chi(s)|$, and the fact that $\chi(s) - s\chi'(s)$ is supported in $(0,2)$.
\end{proof} 
Applying this lemma, we get:
\begin{align}
\log \det (Q + \alpha \chi (Q / \alpha ) ) = \left ( \frac{N}{2\pi} \right ) ^d \int _X \log (q) \dd\mu _d (x) + \mathcal{O} (\alpha ^\kappa ) + \mathcal{O} (N^{d-(1 - 2\rho)}).
\end{align}
Recalling that $(N/2\pi)^d \mathcal{N}^{-1} = \vol(X)^{-1} + \mathcal{O} (N^{-1})$, we get that:
\begin{align}
\log \det (Q + \alpha \chi (Q / \alpha ) ) = (\mathcal{N} + \mathcal{O}(N^{-1})) \strokedint \log (q) \dd\mu_d + \mathcal{O}(N^{d - (1-2\rho)} ). \label{eq:upper}
\end{align}
$\int_X \log (q) \dd\mu_d$ can be uniformly bounded in $z$, so that the $\mathcal{ O} (N^{-1})$ term can be absorbed into $\mathcal{ O} (N^{d - (1-2\rho ) })$. By \eqref{eq:funny bound}, we get the following lower bound by replacing $\alpha$ by $\alpha/4$:
\begin{align}
\log \det (Q + \alpha \chi (Q / \alpha ) ) \ge \mathcal{N} \strokedint \log (q) \dd \mu_d + \mathcal{O} (N^{d - (1 -2\rho)}). \label{eq:lower}
\end{align}

 \begin{lemma}[\textbf{Bound on $\bm A$}]\label{lemma.43}
The number of eigenvalues of $Q$ that are less than $\alpha$ is $\mathcal{ O} (N^d N ^{- \min (2\rho\kappa, (1- 2\rho) )})$. 
 \end{lemma}
 \begin{proof}
 Let $\psi \in C_0^\infty ([-1/2,3/2];[0,1])$ be identically $1$ on $[0,1]$. It then suffices to estimate $\Tr (\psi(Q / \alpha ) ) $. By Proposition \ref{thm.functional calc}, $\psi (Q/\alpha ) = T_{N,q_2} + \mathcal{O} (N^{-\infty}) $, where $q_2 \in S_\rho (1)$ with principal symbol $\psi (q/\alpha)$.
 
 Then by Proposition \ref{thm.trace}
 \begin{align}
 \Tr (\psi (Q/\alpha ) ) &= \Tr (T_{N,q_2} + \mathcal{O} (N^{-\infty})) \\
 &= (N/2\pi)^d\int_X \psi (q/\alpha ) \dd\mu _d (x) + \mathcal{O} (N^{d-(1-2\rho) }) \\
 &\lesssim N^d \alpha^\kappa + N^{d - (1-2\rho )} = \mathcal{O} (N^d N^{-\min(2\rho \kappa , 1 -2\rho)}).
 \end{align}
 \end{proof}

Therefore, putting everything together, we get that
\begin{align}
\log |\det \mathcal{P}^0 | &= \frac{1}{2} \log ( | \det \mathcal{P}^0|^2 ) = \frac{1}{2} \log (\alpha ^{-A} \det 1_\alpha (Q)) = \frac{A}{2}\log (1/\alpha) + \frac{1}{2} \log \det (1_\alpha Q) ).
\end{align}
\eqref{eq:upper} and \eqref{eq:lower} provide upper and lower bounds of $2^{-1}\log \det (1_\alpha (Q))$. Then using that $2 ^{-1}\log q_0 = |f_ 0 -z|$ and Lemma \ref{lemma.43} we get:
\begin{align} 
\left | \log |\det \mathcal{P}^0 | -\mathcal{N} \strokedint _{X} \log | f_0 -z | d\mu _d \right | &\lesssim A \log (1/\alpha ) + \alpha ^{\kappa} + N^{d-(1 - 2\rho)} \\
&\lesssim N^{d - \min (2\rho \kappa, (1- 2\rho))} \log (N)+ N^{-2 \rho \kappa} + N^{d - (1-2\rho)} \\
&\lesssim N^{d - \min (2\rho \kappa, (1- 2\rho))} \log (N).
\end{align}

Recall $\mathcal{N} B_1 = \log |\det \mathcal{P} ^0| - \mathcal{N} \strokedint \log |z - f_0 (x ) | \dd \mu_d $, so that
\begin{align}
B_1 = \mathcal{O} ( N^{-\min(2\rho\kappa,(1-2\rho)) } \log (N) ).
\end{align}
\end{proof}
\section{Summability of $\mathcal{A}_N$}\label{section:last}
Recall that $\mathcal{A}_N = \set{|B(N)| > \e_N}$, where $B(N)= B_1 + B_2 + B_3 $ with:
 \begin{align}
 B_1 &= \mathcal{N}^{-1} \log |\det \mathcal{P} ^0| - \strokedint \log |z - f_0 (x ) | \dd \mu_d (x), \\
 B_2 &= \mathcal{N}^{-1} (\log |\det \mathcal{P}^\delta | - \log |\det \mathcal{P}^0| ), \\
B_3 & = \mathcal{N}^{-1} \log | \det E_{-+}^\delta | .
 \end{align}
The following table summarizes the bounds on $B_1,B_2,$ and $B_3$.
\begin{center}
\begin{tabular}{|c|c|c|}
 \hline 
Bound & Probability of Bound & Reference \\
 \hline 
 $B_1 =\mathcal{O} ( N^{-\min(2\rho\kappa,(1-2\rho)) } \log (N) )$ & $1$ & Claim \ref{claim.bound B_1} \\ 
 \hline 
 $B_2 = \mathcal{O}(\delta \alpha^{-1/2} \mathcal{N}^{1/2} )$ & $> 1 - \exp(-\mathcal{N})$ & Claim \ref{claim.b2 bound} \\ 
 \hline 
 $B_3 \ge \mathcal{N}^{-1}A \log (t \delta ) $ & $> 1 - C \mathcal{N}t^2 - \exp(-\mathcal{N})$ & Claim \ref{claim_bound b_3} \\ 
 \hline 
 $B_3 < 0 $ & $> 1 - \exp(-\mathcal{N})$& Claim \ref{claim_bound b_3} \\ 
 \hline 
 \end{tabular} 
\end{center}
 
Recall that $\rho \in (0,\min(1/2,\e))$ and $\alpha = N^{-2\rho}$. Theorem \ref{thm.main} will follow if $\sum \P(\mathcal{A}_N) <\infty$ for $\e_N = N^{-\gamma}$. Recall that $\delta = \mathcal{O}( N^{-d/2 - \e}) = \mathcal{O}( N^{-d/2} \alpha^{1/2}) $. Fix $0< \gamma < \min (\e-\rho , 2\rho \kappa, 1- 2\rho ) $.

Then $\P(\mathcal{A}_N ) = \P(B > N^{-\gamma}) + \P(B < - N^{-\gamma})$. The first term is:
\begin{align}
\P(B > N^{-\gamma}) = \P(B_3 > N^{-\gamma} - B_2 - B_1 ) .
\end{align}
Because $\gamma <\e - \rho$ and $B_2 = \mathcal{O} (N^{\rho-\e})$ (with overwhelming probability), we see that $B_2 = \mathcal{O}( N^{-\gamma})$ (with overwhelming probability). Similarly, because of the bound on $B_1$ and the choice of $\gamma$, $B_1 = \mathcal{O}( N^{-\gamma})$. So if $N$ is sufficiently large, $N^{-\gamma} - B_2 - B_1 \ge C N^{-\gamma} > 0$. But then by Claim \ref{claim_bound b_3}, $\P(B> N^{-\gamma}) \le e^{-N^d}$ for $N \gg 1$.

Similarly, for $N$ sufficiently large, there exists $C_0 \in (0,1/2)$ such that, $|B_1| + |B_2 | <C_0 N^{-\gamma}$, so $\P(B < -N^{-\gamma}) \le \P(B_3 < -(1-C_0) N^{-\gamma}) = 1- \P(B_3 \ge -(1-C_0)N^{-\gamma})$. By the choice of $\gamma$, bound on $A$ from Lemma \ref{lemma.43}, and selecting $t = \mathcal{N} ^{-2/d - 1/2}$, we get for large enough $N$: $-(1-C_0)N^{-\gamma }\le \mathcal{N}^{-1} A \log (\delta t)$ as long as:
\begin{align}
-N^{-\gamma}(1-C_0) \le \mathcal{N}^{-1} A \log (\delta ).
\end{align}
This requires that $\delta \gg e^{-N^{\beta}}$ for $\beta = \min(2\rho \kappa, 1-2\rho) - \gamma \in (0,1)$. In this case, by Claim \ref{claim_bound b_3},
\begin{align}
\P(B_3 > - N^{-\gamma} ) &\ge \P(B_3 > A \mathcal{N}^{-1} \log (\delta t) ) \\
&\ge 1 - C \mathcal{N} t^2 - e^{- \mathcal{N}} \\
&= 1- C \mathcal{N}^{-2/d} + e^{-\mathcal{N}}.
\end{align}
Therefore $\P(B < - N^{-\gamma}) \le CN^{-2} + e^{-N^{d}}$ for $N \gg 1$.

With this, $\sum_{N=1}^{\infty} \P(\mathcal{A}_N ) = C + \sum _{N \gg 1} \P(A_N) \le C + \sum _{N \gg 1} (N^{-2} + 2 e^{-N^d} ) < \infty$ which proves Theorem \ref{thm.main}.

Note that if $\e > (2 ( \kappa +1))^{-1}$, then we can select $\rho =  (2 (\kappa +1 ) )^{-1}$ and choose $\gamma $ arbitrarily small, so that $\beta = \kappa (\kappa + 1 )^{-1} - \gamma $. While if $\e < (2(\kappa +1 ))^{-1}$, then the maximum $\beta$ can be is $2\epsilon\kappa $. Therefore we have:
\begin{align}
\beta < \begin{cases}
 2 \epsilon \kappa & \text{if } \epsilon < \frac{1}{2(\kappa + 1)}\\
\frac{\kappa }{\kappa + 1} & \text{if } \epsilon\ge \frac{1}{2(\kappa + 1)}
\end{cases}
\end{align}

\section{General random perturbations}\label{section:general random perturbations}

In this section we provide a discussion about how to modify the proof of Theorem \ref{thm.main} (Gaussian random perturbations) to prove Theorem \ref{theorem:general perturbations} (more general random perturbations). We also deduce Theorem \ref{thm:weyl} (stated in the introduction) from Theorem \ref{theorem:general perturbations}.

\begin{proof}

Under the assumptions of $\mathcal{W}_\omega$ (see Definition \ref{def:random}), we have the following probabilistic norm bound:
\begin{align}
\E[\norm{\mathcal{W}_\omega} ^2 ] = \sum _{i,j = 1}^\mathcal{N} \E[ |(\mathcal{W}_\omega) _{i,j}|^2] = \mathcal{O}(\mathcal{N}^2), \label{eq.general_norm_bound}
\end{align}
as well as the following anti-concentration bound (from \cite[Theorem 3.2]{tao2009}): for $\gamma_0 \ge 1/2$, $A_0 \ge 0$, there exists a $c>0$ such that if $M$ is a deterministic matrix with $\norm{M} \le \mathcal{N}^{\gamma_0}$ then
\begin{align}
\P (s_\mathcal{N} ( M + \mathcal{W}_\omega) \le \mathcal{N}^{-(2 A_0  +1 ) \gamma_0 }) \le c \left ( \mathcal{N}^{-A_0 + o(1)} + \P (\norm{\mathcal{W}_\omega} \ge \mathcal{N}^{\gamma_0}) \right )  . \label{eq:general_anti_concentration}
\end{align}
Recall, for an $N\times N$ matrix $A$, we denote $s_1 \ge s_2 \ge \cdots \ge s_N(A) $ the singular values of $A$.

From \eqref{eq.general_norm_bound}, and Markov's inequality, we get
\begin{align}
\P ( \norm{\mathcal{W}_\omega} \ge  N^{d-1} ) = \mathcal{O}( N^{-2})
\end{align}
therefore if $\delta =  N^{-d}$ then $\delta \norm{\mathcal{W}_\omega} = \mathcal{O} (N^{-1})$ with probability at least $1 -C N^{-2}$. From this, Claim \ref{claim:4.2} (the supports of the random empirical measures being contained in a bounded set for $N \gg 1$) will follow by an identical argument.

Next, with probability at least $1 - C N^{-2}$, we have $\delta \norm{\mathcal{W}_\omega}_{•} \alpha^{1/2} \ll 1$. In this event, we can build our perturbed Grushin problem the same way as in Section \ref{grushin approach}.

Next, we have to modify the estimate of $B_2$ which was estimated in Claim \ref{claim.b2 bound}. For this, we simply modify \eqref{eq.3} with a weaker estimate on the probability $\norm{\mathcal{W}_\omega}$ is small. Specifically, we see there exists $C> 0$ such that
\begin{align}
\P(B_2  = \mathcal{O} (\alpha^{-1/2} N^{-1}) )> 1 - C N^{-2}.
\end{align}

The final modification is in estimating $B_3 = \mathcal{N}^{-1} \log |\det E_{-+}^\delta |$. We see, by the same argument presented in Section \ref{grushin approach}, that
\begin{align}
\P ( B_3 < 0 ) \ge 1 - CN^{-2}.
\end{align}
To prove a lower bound, we go through the same argument, to get that:
\begin{align}
\log |\det E_{-+}^\delta | \ge A \log |s_{\mathcal{N} } (T_N f - z + \delta \mathcal{W}_\omega )|.
\end{align}
Next, let
\begin{align}
K_0 := \sup _{z \in \Lambda} \norm{T_N f  -z } = \mathcal{O}(1)
\end{align}
(recall $\Lambda$ is a neighborhood of $f(X)$). By \eqref{eq:general_anti_concentration} (with $\gamma_0 = 1$ and $A_0 = 2$), we have (for $N \gg 1$)
\begin{align}
\P (  s_{\mathcal{N} } (T_N f - z + \delta \mathcal{W}_\omega )  \le N^{-7d} ) &= \P ( s_{\mathcal{•N}} ( \delta^{-1} K_0^{-1} (T_Nf -z )+ K_0^{-1} \mathcal{W}_\omega) \le (N^d)^{-(2A_0 +1)\gamma_0}) \\
&\le c( N^{-2d+ o(1)} + \P(\norm{K_0 ^{-1}\mathcal{  W}_\omega}_{•} \ge N^{-d})  ) \\
&\le c N^{-2}.
\end{align}
Here we use that $\norm{\delta^{-1} K_0^{-1} (T_N f - z)} \le  N^{d}$. With this, we can proceed as in Section \ref{section:last}, with weaker probabilistic estimates. We choose $\rho \in (0,1/2)$, and $0 < \gamma < \min (2\rho \kappa, 1 - 2\rho)$. Writing $\P (\mathcal{A} _N ) = \P(B > N^{-\gamma} ) + \P (B < - N^{-\gamma})$, we see that
\begin{align}
\P ( B> N^{-\gamma}) \le CN^{-2}
\end{align}
for $N \gg 1$. Similarly, in the event $s_\mathcal{N} (T_N f - z + \delta \mathcal{W}_\omega ) \ge N^{-7d}$, we have (for $N \gg 1$)
\begin{align}
A \log |s_\mathcal{N} (T_N f - z + \delta \mathcal{W}_\omega )|  \le N^{d - \gamma}
\end{align}
so that 
\begin{align}
\P ( B_3 >  - N^{-\gamma} )  \ge \P (B_3 > A \mathcal{N}^{-1} \log |s_\mathcal{N} (T_N f - z + \delta \mathcal{W}_\omega )|)\ge 1 - C N^{-2}.
\end{align}
Therefore $\P(B < -N^{-\gamma}) \le C N^{-2}$ for $N \gg 1$. With this, $\sum_1^\infty \P ( \mathcal{A}_N) < \infty$, and we have almost sure weak convergence of the empirical measures of $T_N f + \delta \mathcal{W}_\omega$ to $\vol(X) ^{-1} (f_0)_* \mu_d$.
\end{proof}

\begin{prop}
Theorem \ref{theorem:general perturbations} implies the probabilistic Weyl law (Theorem \ref{thm:weyl}) stated in the introduction.
\end{prop}
\begin{proof}
For $\Lambda \subset \C$ given in the hypothesis, let $A_N =( \vol(X) / \mathcal{N} )\# \set{\Spec(T_N f + N^{-d} \mathcal{W}_\omega ) \cap \Lambda}$. It suffices to show that for each $\e > 0$ 
\begin{align}
\P \left ( \limsup_{N \to \infty} |A_N - \mu_d (f \in \Lambda) | >\e \right ) =0.
\end{align}
We may assume $\Lambda$ is bounded. If not, let $\tilde \Lambda $ be an open, bounded neighborhood of $f(X)$. Recall that almost surely $\Spec(T_N f + \delta \mathcal{W}_\omega) \subset \tilde \Lambda$ for $N \gg 1$. Therefore if $\tilde A_N = ( \vol(X) / \mathcal{N} )\# \set{\Spec(T_N f + N^{-d} \mathcal{W}_\omega ) \cap \Lambda \cap \tilde \Lambda}$, then
\begin{align}
\P \left ( \limsup _{N \to \infty} |A_N - \mu_d (f \in \Lambda)| > \e \right ) = \P \left ( \limsup _{N \to \infty} |\tilde A_N - \mu_d (f \in \Lambda)| > \e\right ).
\end{align}
Now relabel $\Lambda \cap \tilde \Lambda$ as $\Lambda$. Let $\phi , \psi \in C_0^\infty (\C ; [0,1])$ be such that $\supp \phi \subset \Lambda$, $\phi(x) \equiv 1$ for $\dist (x,\partial \Lambda) > \e$, $\psi (x) \equiv 1$ for $x \in \Lambda$, and $\psi (x) = 0$ for $\dist (x,\partial \Lambda) > \e$ (here $\partial \Lambda$ is the boundary of $\Lambda$). Therefore we have
\begin{align}
\frac{\vol(X)}{\mathcal{ N} }\sum_{j=1}^{\mathcal{ N} } \phi(\lambda_i ) \le A_N \le \frac{\vol(X)}{\mathcal{ N} }\sum_{j=1}^{\mathcal{ N} }\psi(\lambda_i).\label{eq:lowup}
\end{align}
By Theorem \ref{theorem:general perturbations}, the lower bound of \eqref{eq:lowup} convergences almost surely to
\begin{align}
\int_\C \phi(z) (f_* \mu_d)(\dd z) = \mu_d (f \in \Lambda) + \mathcal{O} (\e^\kappa).
\end{align}
And similarly the upper bound of \eqref{eq:lowup} converges almost surely to $\mu_d (f \in \Lambda) + \mathcal{O} (\e^\kappa)$ (where the constant in $\mathcal{ O} (\e^\kappa)$ is deterministic). Therefore there exists $C > 0$ such that
\begin{align}
\P \left ( \limsup_{N \to \infty} |A_N - \mu_d (f \in \Lambda) | > C \e^\kappa \right ) =0.
\end{align}
Because $\e > 0 $ is arbitrary, this implies $A_N$ converges almost surely to $\mu_d (f \in \Lambda)$. Then, because $\mathcal{ N} = \vol(X) (N/2\pi)^d + \mathcal{O} (N^{d-1})$, $(N/2\pi)^d \vol(X) \mathcal{N}^{-1} A_N$ converges almost surely to $\mu_d (f\in \Lambda)$.
\end{proof}

\smallsection{Acknowledgements} 
The author is grateful to Maciej Zworski for suggesting this problem and many helpful discussions, to Martin Vogel for helpful insights and catching many errors in an earlier draft, and to an anonymous referee for several helpful suggestions. This paper is based upon work jointly supported by the National Science Foundation Graduate Research Fellowship under grant DGE-1650114 and by grant DMS-1952939. 
\bibliographystyle{siam} 
\bibliography{references}

\end{document}